\documentclass[12pt,twoside]{amsart}
\usepackage{bbm}
\usepackage{enumerate}
\usepackage{subcaption}
\usepackage{xspace}

\usepackage{graphicx,color}
\usepackage[all]{xy}
\usepackage{verbatim}
\usepackage[dvipsnames]{xcolor}
\usepackage{eso-pic}% http://ctan.org/pkg/eso-pic

\theoremstyle{plain}
\newtheorem{The}{Theorem}%[section]
\newtheorem*{The*}{Theorem}

\newtheorem{Lem}{Lemma}

\newtheorem*{Cor*}{Corollary}

\theoremstyle{definition}

\newtheorem{Rem}{Remark}

\newtheorem*{Rem*}{Remark}

\numberwithin{equation}{section}

\renewcommand{\Im}{\operatorname{Im}}
\renewcommand{\Re}{\operatorname{Re}}

\DeclareMathOperator{\del}{\partial}

\newcommand{\R}{\mathbb{R}}

\newcommand{\C}{\mathbb{C}}
\newcommand{\N}{\mathbb{N}}
\newcommand{\Z}{\mathbb{Z}}

\renewcommand{\H}{\mathbb{H}}

\begin{document}

\title{Stability properties of $2$-lobed Delaunay tori in the $3$-sphere }

\author{Lynn Heller}

\address{ Institut f\"ur Differentialgeometrie\\  Leibniz Universit{\"a}t Hannover\\ Welfengarten 1, 30167 Hannover}

 \email{lynn.heller@math.uni-hannover.de}

 \author{Sebastian Heller}
\address{Department of Mathematics\\
University of Hamburg\\
20146 Hamburg, Germany
 }
 \email{seb.heller@gmail.com}

\author{Cheikh Birahim Ndiaye}
\address{Department of Mathematics of Howard University\\
204 Academic Support Building B
Washington, DC 20059k \\ USA
 }
 \email{cheikh.ndiaye@howard.edu}

%\subjclass[2010]{Primary 53A05, 53 A 30, 53C42; Secondary 37K15}

%\subjclass{53A10,53C42,53C43,14H60}
 \date{\today}

%\thanks{Geometry at Infinity, GRK1670}
\noindent

\begin{abstract} 
We show that the of $2$-lobed Delaunay tori are stable as constrained Willmore surfaces in the $3$-sphere.

\begin{center}
\vspace{0.5cm}
\includegraphics[width= 0.19\textwidth]{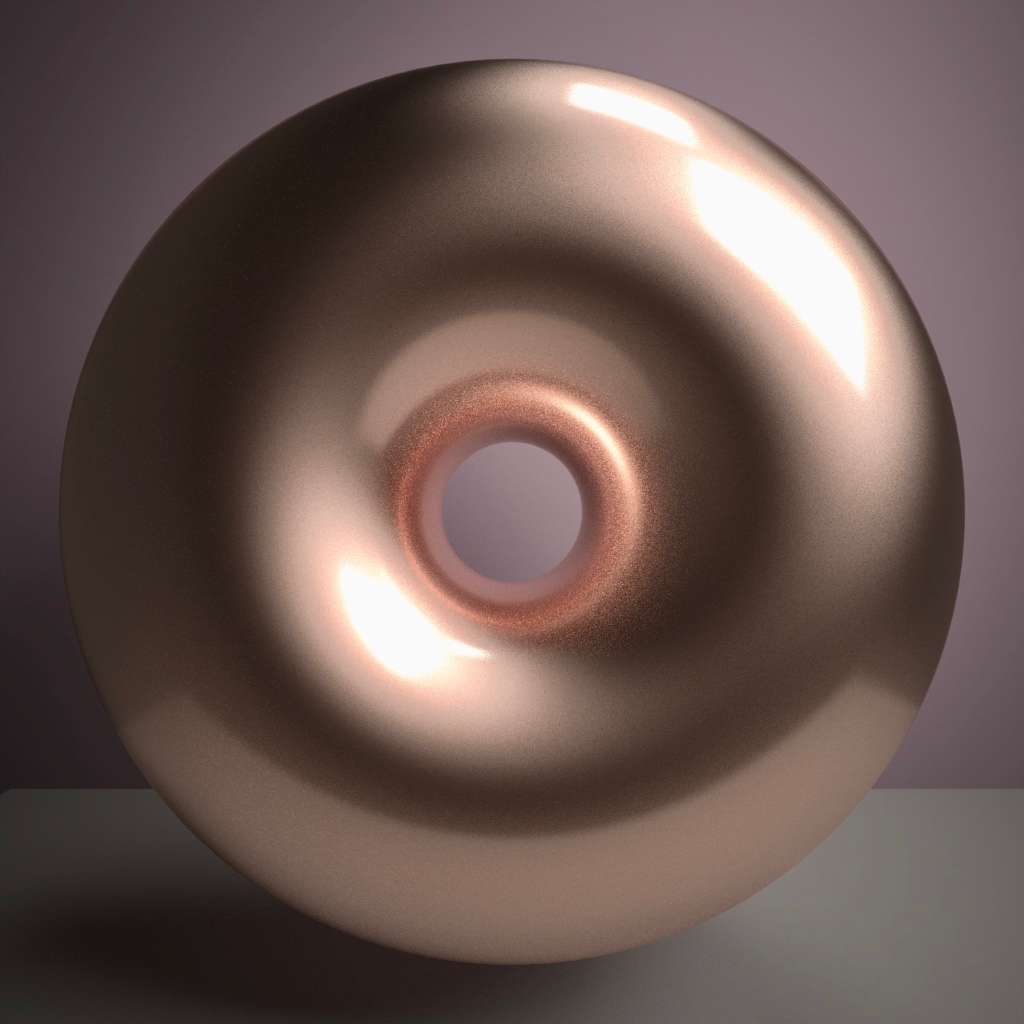}\hspace{0.2cm}
\includegraphics[width= 0.19\textwidth]{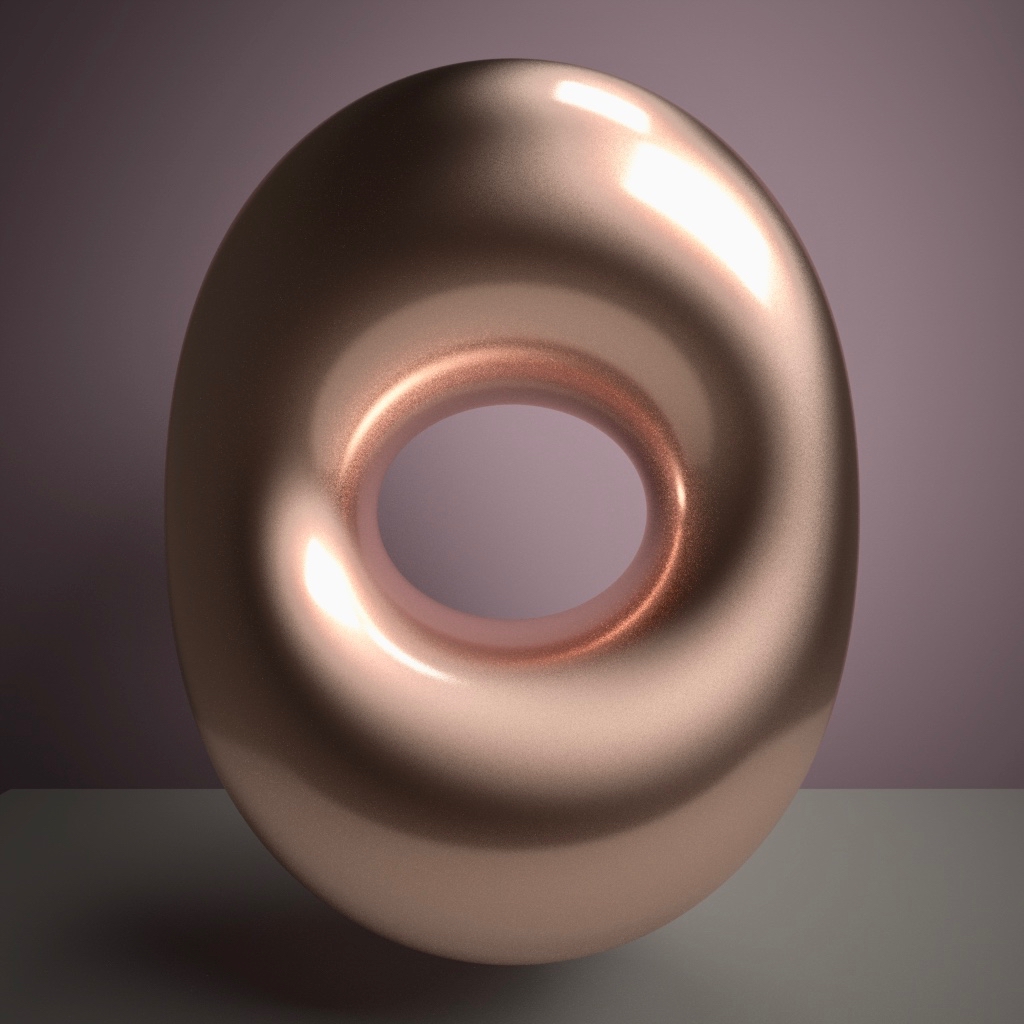}\hspace{0.2cm}
\includegraphics[width= 0.19\textwidth]{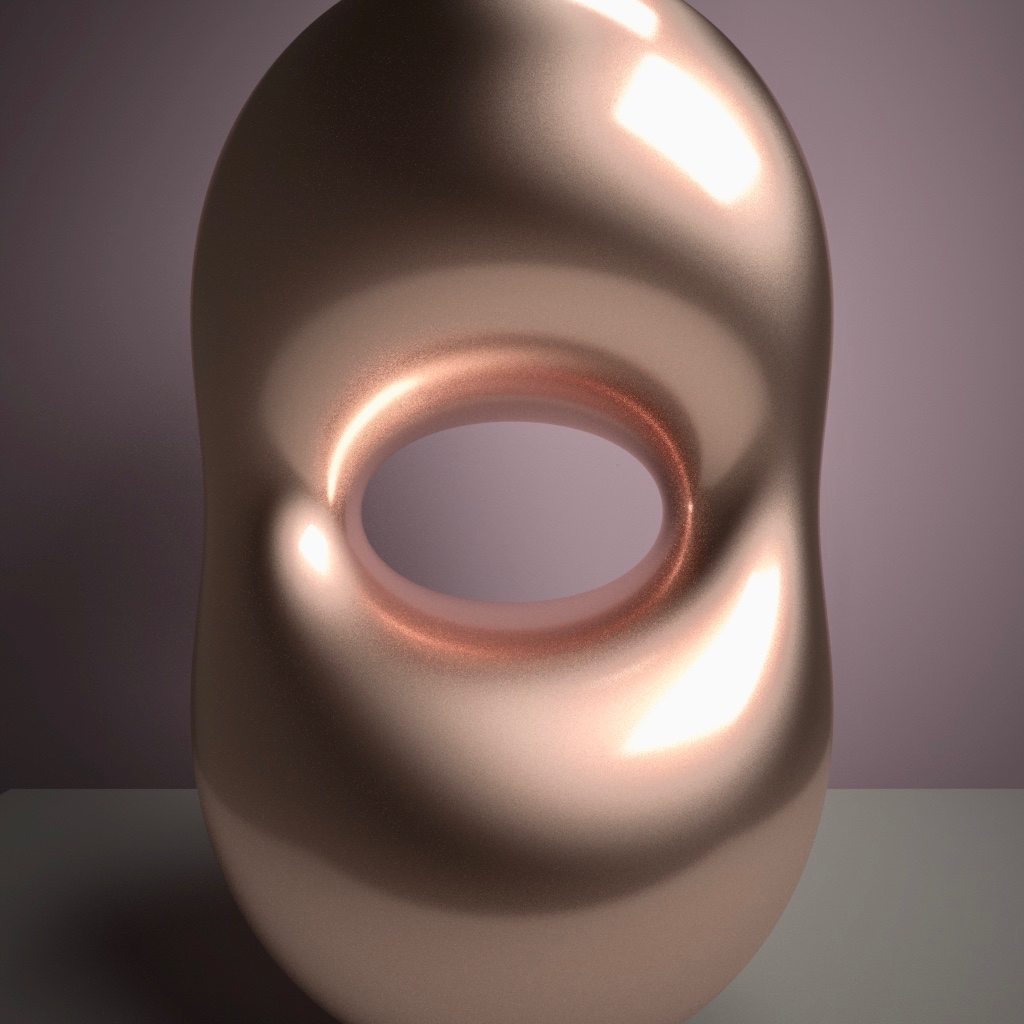}\hspace{0.2cm}
\includegraphics[width= 0.19\textwidth]{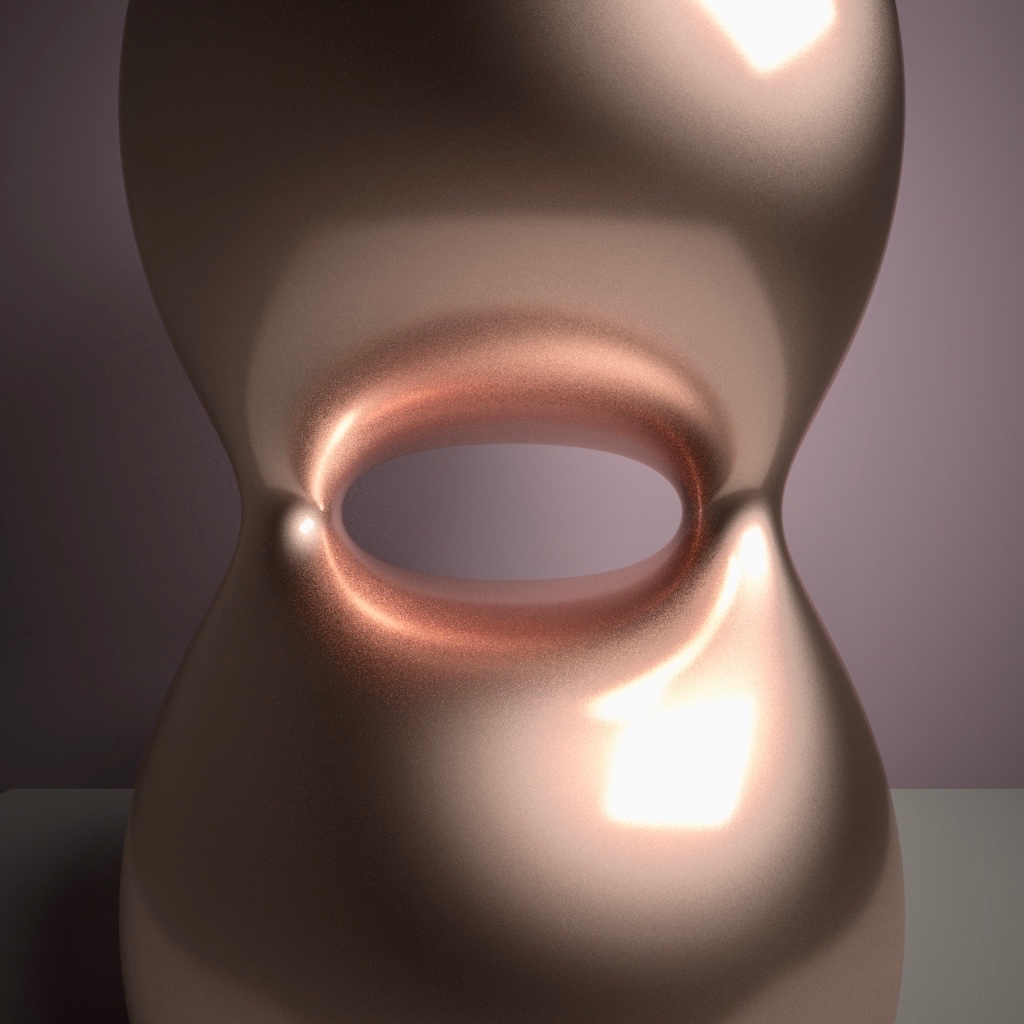}
\vspace{-0.6cm}
\end{center}
%\keywords{ Constrained Willmore tori,  spectral curve, CMC tori}
 \end{abstract}

\maketitle

%\setcounter{tocdepth}{1}
 %\tableofcontents

%%%%%%%%%%%%%%%%%%%%%%%%%%%%%%%%%%%%%%%%%%%%%%%%%%%%%%%%%%%%%%%%%%%%%%%%%%%%%%
%           Introduction                                                     %
%%%%%%%%%%%%%%%%%%%%%%%%%%%%%%%%%%%%%%%%%%%%%%%%%%%%%%%%%%%%%%%%%%%%%%%%%%%%%%

\section{Introduction}
We study conformal immersions $f\colon M\to S^3$ from a Riemann surface $M$ into the $3$-sphere that are critical points of the Willmore energy
\[
\mathcal{W}(f)=\int_M (H^2+1) dA
\] 
under conformal variations. Here we denote by $H$ is the  mean curvature and $dA$ is the induced area form of $f$. Geometrically speaking $\mathcal W$ measures the roundness of a surface, physically  the degree of bending, and in biology $\mathcal W$ appears as a special instance of the Helfrich energy for cell membranes.  The conformal constraint augments the Euler-Lagrange equation by a holomorphic quadratic differential \linebreak $\omega\in H^0(K^2_M)$ paired with the trace-free second fundamental  form $\mathring{A}$ of the immersion
\[
\triangle H+ 2H(H^2+1-K)=\,<\omega,\mathring{A}>,
\]
see \cite{BPP}.
The first examples of constrained Willmore surfaces are given by surfaces of constant mean curvature in a $3$-dimensional space form. In this case the critical surface is isothermic: the holomorphic quadratic differential is no longer uniquely determined by the immersion leading to a singularity of the moduli space.  Since there are no holomorphic quadratic differentials on a genus zero Riemann surface, constrained Willmore spheres are the same as Willmore spheres. For  genus $g\geq1$ surfaces this is no longer the case: constant mean curvature (CMC) surfaces (and their M\"obius transforms) are constrained Willmore, as one can see by choosing $\omega=\mathring{A}$ to be the holomorphic Hopf differential, but not Willmore unless the surface is totally umbilic. In the case of $M = T^2$ being a torus  Bohle \cite{Bohle}, partially motivated by the manuscript of Schmidt \cite{Schmidt}, showed that all constrained Willmore tori arise from linear flows on Jacobians of finite genus spectral curves. Starting at the Clifford torus, which has mean curvature $H=0$ and a square conformal structure, these surfaces in the 3-sphere limit with monotone and unbounded mean curvature to a circle and thereby sweeping out  all rectangular conformal structures. Less trivial examples come from the Delaunay tori of various lobe counts in the 3-sphere whose spectral curves have genus 1 (see Figure~\ref{fig:torus-tree}). 
By the solution of the Lawson and the Pinkall-Sterling conjecture, due to Brendle \cite{Brendle} and  Andrews \& Li \cite{AndrewsLi} using Brendle's approach, 
those are the only embedded CMC tori in the 3-sphere.
 \begin{figure}
\centering
\includegraphics[width=0.375\textwidth
]
{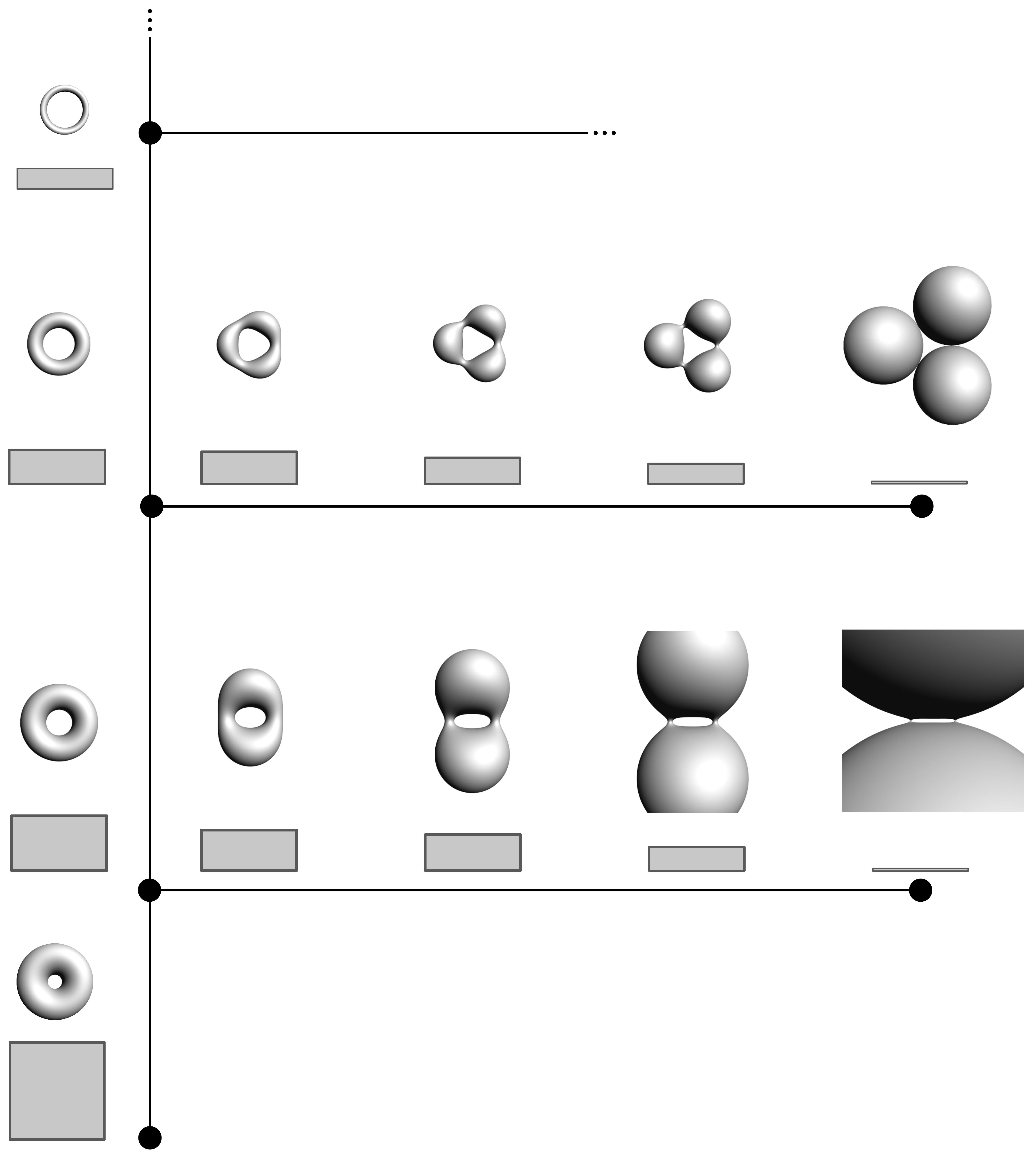}
\caption{
The vertical stalk represents the family of homogenous tori,
starting with the Clifford torus at the bottom.
Along this stalk are bifurcation points at which
the embedded Delaunay tori appear along the horizontal lines.
The rectangles indicate the conformal types. Images by Nicholas Schmitt.
}
\label{fig:torus-tree}
\end{figure}

Existence and regularity of a minimizer $f\colon T^2\longrightarrow S^3$  in a given conformal class for any genus was shown by Kuwert and Sch\"atzle \cite{KuwertSchaetzle} under the provision that the infimum Willmore energy $\mathcal{W}(f)$  is below $8\pi$. This restriction is used to rule out minimizers with branch points. Similar results were proven by  Riviere \cite{Riviere} using the divergence form of the Euler-Lagrange equation. Parallel to the solution of the Willmore conjecture Ndiaye and Sch\"atzle \cite{NdiayeSchaetzle1, NdiayeSchaetzle2} showed that for rectangular conformal classes $(0,b)$ in a neighborhood of the square conformal class $(0,1)$ the homogenous tori $f_H^b$ (whose spectral curves have genus 0) are the unique minimizers for the constrained Willmore problem. In a recent preprint \cite{HelNdi} minimizers of the Willmore energy with conformal class lying in a suitable neighborhood (of the Teichm\"uller space) of the square class have been identified to be equivariant. As a corollary we obtain the real analyticity of the minimal Willmore energy $\omega(a,b)$ for $b \sim 1, $ $b \neq 1$ and $a\sim0^+.$ Furthermore, we obtain in this region that the minimal energy is $\mathcal C^0$ but not $\mathcal C^1$ at rectangular conformal classes. In particular, using the same arguments as in \cite{NdiayeSchaetzle1} together with \cite[Corollary 6]{MontielRos}, the homogenous tori $f_H^b$ with $b \sim 1$ uniquely minimize the penalized Willmore energy
$$\mathcal W_{\alpha} := \mathcal W - \alpha \Pi^1$$
among immersions of conformal class $(a,b),$ with $a \leq \tfrac{1}{2}$, if $\alpha>0$ is small enough.

The homogeneous tori of revolution eventually have to fail to be minimizing since their Willmore energy can be made arbitrarily large. Calculating the 2nd variation of the Willmore energy $\mathcal{W}$ along tori of revolution with circular profiles Kuwert and Lorenz \cite{KuwertLorenz} showed that negative eigenvalues appear at those conformal classes $(0, b_{k-2})$ whose rectangles have side length ratio $\sqrt{k^2-1}$ for $k\geq 2$. These are exactly the rectangular conformal classes from which the $k$-lobed Delaunay tori (of spectral genus 1) bifurcate and the corresponding homogenous torus is of Index $2(k-2)$.  Any of the families $f_k^b$ starting from the Clifford torus, following homogenous tori to the $k$-th  bifurcation point, and branching to the $k$-lobed Delaunay tori which limit to a neckless of spheres, sweep out all rectangular conformal classes (see Figure~\ref{fig:torus-tree}). The Willmore energy varies strictly monotonically along the family  between $2\pi^2\leq \mathcal{W}<4\pi k$ as $b$ varies from $1$ to $\infty$, see \cite{KilianSchmidtSchmitt1, KilianSchmidtSchmitt2}.  For $k= 2$ we obtain an immersion with Willmore energy below $8\pi$ for every rectangular conformal class satisfying the energy bound in \cite{KuwertSchaetzle} showing the existence of an embedded minimizer in all these conformal classes. Thus it is conjectured that $f^b_2$ minimizes the Willmore energy for rectangular conformal classes $(0,b)$ in all codimensions.

In this paper we prove the necessary condition for the conjecture to hold in $3$-space:

\begin{The}\label{stability}
The family of $2$-lobed Delaunay tori $f^b$ are constrained-Willmore-stable for all $b \in \R_{\geq 1}$. Moreover, the kernel dimension of the stability operator is at most $1$ (up to invariance) for $b > b_0$ and reduces to a variation $\varphi$ of the underlying curve. In particular,
$$\delta^2 \Pi^1(f^b) (\varphi, \varphi) = 0.$$

\end{The}

\begin{Rem}
This theorem guarantees that variants of the implicit function theorem, as carried out in \cite{NdiayeSchaetzle1} and \cite{HelNdi}, can be applied to show that all solutions of the Euler-Lagrange equation $W^{4,2}$-close to $f^{\bar b}$ with $\Pi^1$-Lagrange multiplier $0\leq\alpha < \alpha^b$,  coincides with the $f^b$, if $f^{\bar b}$ is stable for $W_{\alpha^b}$.
\end{Rem}

\subsection*{Acknowledgements} The first author is supported by the DFG within the SPP {\em Geometry at Infinity}, and the second author is supported by RTG 1670 {\em Mathematics inspired by string theory and quantum field theory} funded by the  DFG.

\section{Constrained Willmore stability of the $2$-lobe family}\label{sec:stability}

In this section we show that the 2-lobed Delaunay tori $f^b$ $(b\neq b_0)$ are stable as constrained Willmore surfaces. For $b<b_0$ the surface $f^b$  is strictly stable \cite{KuwertLorenz}. At $b =b_0$ the stability operator has a $2$-dimensional kernel spanned by

\begin{equation}
\varphi = \cos(2x)\vec{n}^{b_0} \quad \text{ and } \quad \tilde \varphi = \sin(2x) \vec{n}^{b_0},
\end{equation}
where $x$ is the parameter of the profile curve and  $vec{n}^{b_0}$ is the normal of $f^{b_0}$.
Thus in order to prove Theorem \ref{stability}, it suffices to show the stability (up to invariance) of the second variation of $\mathcal W$ at $f^b$ for $b>b_0$.

For $b>b_0$ the surface $f^b$ arises by rotating its profile curve -- an arc length parametrized elastic curve $\gamma^b $ in the upper half plane $\mathcal H^2$ given by $$\mathcal H^2= \{(u, v) \in \R^2 \;|\; v>0\}$$ -- considered as the hyperbolic plane --  around the $u$-axis. The arc length parametrized closed curve $\gamma^b$ is given by
\[\gamma^b: x\in\R/2\pi b\Z\mapsto (u(x),v(x)),\]
with $L=2\pi b$ being its length. It satisfies 
\[u'(x)^2+v'(x)^2=v(x)^2.\]
 The oriented unit normal in $\mathcal H^2$ is given by
\[\vec{n}_g = -v'(x) \frac{\partial}{\partial u}+u'(x)\frac{\partial}{\partial v}.\]

Instead of the Euclidean $3$-space $\R^3$, we consider $f^b$ mapping into the subspace $\mathcal H^2 \times S^1 \subset \R^3 \subset S^3$ with conformally equivalent metric 
$$g=  \tfrac{1}{v^2} (du\otimes du + dv\otimes dv)  + d \varphi \otimes d\varphi,$$

with singularities for $v =0.$ In these coordinates the corresponding conformally parametrized torus of revolution has a particularly simple form
\[f^b\colon M=\R/2\pi b\Z\times \R/2\pi\Z\longrightarrow \mathcal H^2\times S^1;\;\;  (x,y)\longmapsto (u(x),v(x), y)\]
with $v(x)>0$ for all $x$. This defines a conformal immersion into $\R^3\subset S^3.$ The fibers of the surface, i.e., the curves given by $$y \mapsto f(x_0, y)$$ for a fixed $x_0 \in \R/2\pi b,$ are geodesics
with respect to $g$. 

The Willmore functional as specified in \cite{Chen} is given by
\[\mathcal W(f) = \int_M (H^2-G)dA\]
where $H$ is the mean curvature, $G$ is the determinant of the second fundamental form (both with respect to the induced metric), and $dA$ is the induced area form. With this definition $\mathcal W$ is invariant under conformal changes of the ambient metric. We use $g$ to compute the Willmore functional and its derivatives.

Since $f^b$ is an isometric immersion and the fibers are geodesics we obtain
\[H=\kappa, \quad G=0, \quad  \text{ and } \quad dA=dx\wedge dy,\]
with $\kappa=\kappa(x)$ being the geodesic curvature of $\gamma^b$ in the hyperbolic plane. Thus the conformal invariance of the Willmore functional gives:

$$\mathcal W (f^b) = \int_ {T^2} (\tfrac{1}{4}H_{\R^3}^2 ) dA_{\R^3} = \tfrac{1}{4} \pi \int_0^{2 \pi b} \kappa^2 ds = \tfrac{1}{4} \pi  \mathcal E(\gamma),$$

and $\mathcal E$ is the elastic energy functional of curves in $\mathcal H^2.$

\subsection{Variation of the profile curve}\label{sec:varprof}
Consider normal variations $$\Phi(x,y) = \varphi(x,y) \vec{n}_{g}(x,y),$$
where $ \vec{n}_{g}$ is the unit length normal field of the surface with respect to $g.$
For every fixed $y_0$ we have that $\gamma^b_{y_0} = f^b(x, y_0)$ is an arc length parametrized elastic curve in $\mathcal H^2$, and $\Phi(x, y_0)$ is a normal variation of the curve $\gamma^b_{y_0} $. 

Since the Willmore functional of the surface $\mathcal W(f^b)$ is the energy functional $\mathcal E$ of the curve $\gamma^b_{y_0}$ and the conformal constraint corresponds to the length constraint on curves \cite{Pinkall-Hopf}, we have that the second variation satisfies for every $y_0\in\R/2\pi\Z$

\begin{equation}
\begin{split}\delta^2\mathcal W_{\beta^b} (f^b) (\varphi(x, y_0)\vec{n}_g(x,y)) &= \delta^2 \mathcal E (\gamma_{y_0}) (\Phi(x,y_0)) -  \beta^b \delta^2 \mathcal L (\gamma^b) (\Phi (x, y_0))\\ &=: \delta^2 \mathcal E_{\beta^b} (\gamma_{y_0} )(\Phi(x,y_0)),\end{split}\end{equation}

where $\mathcal L$ is the length functional of curves in $\mathcal H^2$. 

\begin{Lem}\label{curvestability}
Let $f^b$ be the 2-lobed Delaunay torus with conformal class $(0,b)$ and $\gamma^b$ the corresponding elastic curve in $\mathcal H^2.$ Then 
$$\delta^2 \mathcal E_{\beta^b} (\gamma_{y_0}) (\Phi(x,y_0), \Phi(x,y_0)) \geq 0$$

for normal variations $\Phi$ preserving the length constraint.
Moreover, for $b>b_0$ the kernel of $\delta^2 \mathcal E_{\beta^b}$ is at most $1$-dimensional modulo  the isometries of  $\mathcal H^2.$
\end{Lem}
\begin{proof}
For $L = 2 \pi b >0$ fixed, consider the minimization problem 

$$\inf \; \{\mathcal E(\gamma)\; | \; \gamma \text{ is a closed curve in } \mathcal H^2 \text{ with length L }\}.$$

The minimum is attained by an elastic curve $\gamma_{\min}$ \cite{LangerSinger2} for which 

$$\delta^2 \mathcal E_{\beta^b} (\gamma_{\min}) (\Phi(x,y_0), \Phi(x,y_0))\geq 0.$$

Rotating $\gamma_{\min}$ we obtain an embedded and isothermic constrained Willmore torus $f^{\min}$ in $S^3.$ The Willmore energy of $f^{\min}$ must be below $8 \pi,$ by using $f^b$/$\gamma^b$s as competitor.  Thus by 
\cite[Theorem 1]{HHNd1}  and \cite{Richter}, $f^{\min}$ is CMC in $S^3.$  The classification of embedded CMC tori \cite{Brendle, AndrewsLi} and \cite{KilianSchmidtSchmitt1} shows that $f^{\min} = f^b$ for the respective $b >1$ up to invariance. Therefore, we have $\gamma_{\min} = \gamma^b$ up to isometry.

Since $\delta^2 \mathcal E_{\beta^b} \geq0$, kernel elements consists of eigenvectors with respect to eigenvalue 0 of $\delta^2 \mathcal E_{\beta^b}.$ This gives rise to a 4th order linear ODE, see \cite{LangerSinger2}, which has  at most $4$-dimension worth of solutions, three of which corresponds to the 3-dimensional space of isometries of $\mathcal H^2.$ 
Thus up to invariance of the equation the kernel is at most one-dimensional.
\end{proof}

\begin{Rem}
The Willmore energy of $f^b$ is shown to be monotonically increasing \cite{KilianSchmidtSchmitt1}. The Lagrange multiplier of $f^b$ (which corresponds to its constant mean curvature in $S^3$)  is shown to be strictly increasing for $b < b_0$ and strictly decreasing for $b>b_0$ \cite{KilianSchmidtSchmitt1}. \end{Rem}

\subsection{General variations}
 
In order to compute the full second variation $\delta^2 \mathcal W_{\beta^b},$ we first explicitly parametrize a normal variation $F_t^b$ of $f^b$, compute the first and second fundamental form and then determine the Taylor expansion of the Willmore energy at $t=0$ to the second order. The second order term is then

$$\left(\frac{d}{d t}\right)^2|_{t=0} \mathcal W(F_t^b ) = D^2 \mathcal W(f^b) (\del_tF_t	^b, \del_tF_t	^b) + D \mathcal W (\del^2_{tt}F_t^b).$$

A short computation shows that the following holds for a variation $F_t^b$ of a constrained Willmore torus:
$$D W(\del^2_{tt} F_t^b) = -\beta^bD^2 \Pi^2 (\del_t F_t^b \del_t F_t^b) + \beta^b \left(\frac{d}{d t}\right)^2|_{t=0}(\Pi^2(F_t^b)).$$

A constrained Willmore torus is stable for $\mathcal W$ if 
$$\left(\frac{d}{d t}\right)^2|_{t=0}\mathcal W(F_t^b ) \geq 0$$
for all variations $F_t^b$ preserving the conformal type. In this case $$\left(\frac{d}{d t}\right)^2|_{t=0}(\Pi^2(F_t^b)) \equiv0.$$To determine whether a surface is stable we will show that 
\begin{equation}\label{2W-2P}
\left(\frac{d}{d t}\right)^2|_{t=0} \mathcal W(F_t^b )- \beta^b \left(\frac{d}{d t}\right)^2|_{t=0}(\Pi^2(F_t^b)) \geq 0
\end{equation}

for all variations $F_t^b.$ 

Similar arguments as in Section \ref{sec:varprof} show that  $\delta^2 \mathcal W_{\beta^b} = \delta^2 \mathcal W + \beta^b \delta^2 \Pi^2$ can be decomposed into

$$\delta^2 \mathcal W_{\beta^b} (f^b)  = \delta^2 \mathcal E_{\beta^b} + \mathcal Q,$$

where $\delta^2 \mathcal E_{\beta^b}$ is tensorial in $y,$ i.e., for normal variations of the form 
$$\Phi(x,y) = p(x) q(y) \vec{n}_{\tilde g}$$ we have

$$\delta^2 \mathcal E_{\beta^b} (f^b)(\Phi(x,y), \Phi(x,y)) = \delta^2 \mathcal E_{\beta^b}(f^b)(p(x) \vec{n}_{\tilde g}, p(x) \vec{n}_{\tilde g}) q^2(y).$$

The term $\mathcal Q$ in this decomposition has no 0th order term in $y.$  We want to show that $\mathcal Q$ is strictly positive unless the normal variation given by $\varphi$ is a M\"obius variation. Together with Lemma \ref{curvestability} this implies Theorem \ref{stability}. We further  split  $\mathcal Q$ with respect to \eqref{2W-2P}
$$ \mathcal Q= \mathcal Q_1 + \mathcal Q_2$$
where $\mathcal Q_1$ corresponds to $\left(\frac{d}{d t}\right)^2|_{t=0} \mathcal W(F_t^b )$ and $\mathcal Q_2$ corresponds to $\left(\frac{d}{d t}\right)^2|_{t=0}(\Pi^2(F_t^b)).$

\subsection{Explicit formulas}
For the 2-lobed Delaunay torus $f^b$ consider normal variations determined by 
$$\Phi(x,y) = \varphi(x,y) \vec{n}_g,$$

for a function $\varphi \colon T^2 \to\R$, i.e.,
\[t\mapsto F_t^b\colon T^2\to \H^2\times S^1; (x,y)\mapsto (u(x)-t v'(x) \varphi(x,y) ,v(x)+t u'(x) \varphi(x,y), y).\]
$F_t^b$ is an immersion for every $t\sim 0$, although $F_t^b $ is generally not conformal for $t\neq0$.

For $t\sim0$ the induced family of metrics $g_t$ of $F_t^b$ on the torus 
\[T^2_b=\R/(2\pi b\Z)\times \R/(2\pi \Z)\]
are given by:

\[g_t=dx\otimes dy+dy\otimes dy+t(\alpha dx\otimes dx)+t^2(\beta dx\otimes dx+\gamma(dx\otimes dy+dy\otimes dx)+\delta dy\otimes dy)+\dots,\]
where
\begin{equation}\label{metricexpansion}
\begin{split}
\alpha&=-2\kappa\varphi\\
\beta&=\varphi_x^2+\kappa^2\varphi^2+\frac{v'^2}{v^2}\varphi^2+\frac{2}{v}(\kappa u'\varphi^2+v'\varphi\varphi_x)\\
\gamma&=\frac{v'}{v}\varphi_y\varphi+\varphi_x\varphi_y\\
\delta&=\varphi_y^2.
\end{split}
\end{equation}
The volume forms are  $dA_t = \sqrt{\det(g_t)}.$ and the second fundamental form at $t=0$ is given by 
\begin{equation}
II_0 = \begin{pmatrix} \kappa&0\\0&0\end{pmatrix} .
\end{equation}

The equation for the geodesic curvature $\kappa$ of $\gamma^b$ gives
\[u''(x)=-v'(x)\kappa(x)+\frac{2 u'(x) v'(x)}{v(x)}\;\;\text{ and } v''(x)=u'(x) \kappa(x)- \frac{u'(x)^2-v'(x)^2}{v(x)}.\]

\subsubsection{The second derivative of $\mathcal W(F_t^b)$}
We expand the real function $t \mapsto W(F_t^b)$ at $t=0$ up to second order. 
For convenience, the dependencies of the involved functions are suppressed, and $\varphi_{x}:=\frac{\partial  \varphi}{\partial x},\dots\;$. 
Then a lengthy but straight forward computation shows that

\[D\mathcal W_f(\varphi):=\frac{d}{dt}\mid_{ t=0}\mathcal W(F_t^b)=-\frac{1}{4}\int_M \left(2\kappa''+\kappa^3-2\kappa\right) \varphi \;dx \wedge dy\]
and the second derivative splits into
\[\left(\frac{d}{dt}\right)^2{\mid _{t=0}}\mathcal W(F_t^b)= \int_M\left(\mathcal A + \mathcal Q_1(\varphi)\right)\varphi \;dx\wedge dy,\]
where $\mathcal A$ tensorial in $y$ (in the above sense) and $\mathcal Q_1$ is given by

\begin{equation}
\mathcal Q_1(\varphi)=\left(\tfrac{1}{8}\kappa^2 +\tfrac{1}{2}\right)\varphi_{yy}+\tfrac{1}{4} \varphi_{yyyy}+\tfrac{1}{2} \varphi_{xxyy}.
\end{equation}

\subsubsection{Second derivative of the conformal type}
We compute the change of the induced conformal structure for normal variations. The Teichm\"uller space of Riemann surfaces of genus 1 is parametrized by the modulus $\tau$ 
\[\tau=\frac{\int_{\alpha_1}\omega}{\int_{\alpha_2}\omega}\]
for a non-zero holomorphic 1-form $\omega$ and an appropriate choice $\alpha_1,\alpha_2$ of generators of the first fundamental group such that $\Im{\tau}>0.$ 
In the following, we choose $\alpha_1,\alpha_2$ such that
\begin{equation}\label{taunormali}\tau\in\left\{c\in\mathbb C\mid -\frac{1}{2}<\Re{c}\leq \frac{1}{2};\Im{c}>0, c\bar c\geq1\right\}.\end{equation}

Consider a family of 1-forms with respect to the metric $g_t$ (given by \eqref{metricexpansion})
\begin{equation}
\begin{split}\omega_t=&dx+ i dy+t\left(\left(\frac{\alpha}{2}+p\right)dx+ip dy\right)\\&+t^2\left(\left(-\frac{1}{8}\alpha^2+\frac{1}{2}\beta+i\gamma-\frac{1}{2}\delta+\frac{1}{2}\alpha p+q\right)dx+i q dy \right)\\
=&\left(1+t p+t^2q\right) dz+t \frac{\alpha}{2} dx+t^2\left(-\frac{1}{8}\alpha^2+\frac{1}{2}\beta+i\gamma-\frac{1}{2}\delta+\frac{1}{2}\alpha p\right)dx\end{split}\end{equation}
for arbitrary functions $p,q\colon T^2\to\mathbb C.$
Let $J_t$ be the rotation by $\frac{\pi}{2}$ of $g_t$, then  
\[J_t^*\omega_t= i\omega_t+O(t^3).\]
Hence, $\omega_t$ is a $(1,0)$-form up to second order in $t.$ We want to determine the functions $p$ and $q$ such that 
$\omega_t$ is closed (hence holomorphic) up to second order in $t$. 

In first order we obtain the equation
\[d(p\;dz)=-d\;\frac{\alpha}{2}\wedge dx=-i\frac{\alpha_y}{4}\;d\bar z\wedge dz.\]
The change of conformal type is given by the change of the ratio of the periods of $\omega_t$ along
the generators 

\[(2\pi b, 0) \text{ and } (0,2\pi)\]

of $\pi_1(T^2).$ It is most convenient to choose a function $p$ such that the integral of $\omega_t$ along $(0,2\pi)\in\pi_1(T^2)$ is independent of $t$. This can be achieved to first order by 
choosing a solution $p$ of the equation
\begin{equation}\label{dbarpay}\bar\partial p=-i\frac{\alpha_y}{4}\; d\bar z\end{equation}
that is perpendicular (with respect to the $L^2$ inner product of the area form $\frac{i}{2}dz\wedge d\bar z$) to the constant functions. In this case there exists a unique solution of \eqref{dbarpay} by Serre-duality since 
\[\int\alpha_y \; d\bar z\wedge dz=0\]
by Stokes.
In fact, we then obtain

\[\int_{(0,2\pi)}p\; dz=i\int_{(0,2\pi)}p \;dy=\frac{i}{2\pi b}\int_{T^2} p \;dx\wedge dy=0,\]

where the second to last equality is due to Fubini and the fact that

\[p\;dz+\frac{\alpha}{2}\;dx\]

is closed. Hence, due to \eqref{taunormali} the change of conformal type of $g_t$
is given by

\[\tau_t=b\; i+t\frac{i}{2\pi}\int_{(2\pi b,0)}\left(\frac{\alpha}{2}\;dx +p \;dz\right)+O(t^2).\]

Note that
\[\int_{(2\pi b,0)}\left(\frac{\alpha}{2}\;dx +p \;dz\right)=\frac{i}{2\pi}\int_{T^2 } \left(\frac{\alpha}{2} +p\right)dz\wedge d\bar z=\frac{i}{2\pi}\int_{T^2} \frac{\alpha}{2} \; dz\wedge d\bar z=\frac{1}{2\pi}\int_{T^2} \alpha \; dx\wedge dy\] is real as $\alpha$ is real.
Also note that we are only interested in normal variations which do not change the conformal type in first order, i.e. we impose that
\begin{equation}\label{gtconfcons}
0=\frac{1}{2\pi}\int_{T^2} \alpha \; dx\wedge dy=-\frac{1}{\pi}\int_{T^2}\kappa\varphi \; dx\wedge dy.\end{equation}

To compute the change of the conformal type to second order, we proceed as for the first order variation and
take the unique solution $q\colon T^2\to\C$  of
\[\bar\partial q=\frac{\partial}{\partial y}\left(-\frac{1}{8}\alpha^2+\frac{1}{2}\beta+i\gamma-\frac{1}{2}\delta+\frac{1}{2}\alpha p\right) d\bar z\]
perpendicular to the constant functions. Analogously to the first order computations, and under the conformal constraint \eqref{gtconfcons}, the conformal type of the induced metric $g_t$ changes in second order as
\[\left(\frac{d}{dt}\right)^2\mid_{t=0} \tau_t = \frac{i}{\pi}\int_{(2\pi b,0)} \left(-\frac{1}{8}\alpha^2+\frac{1}{2}\beta+i\gamma-\frac{1}{2}\delta+\frac{1}{2}\alpha p\right)dx,\]
which gives us (using the first order constraint  and  \eqref{metricexpansion})
\begin{equation}
\begin{split}
\tau_t=&b i+t^2\frac{i}{4\pi^2}\int_{T^2}\left(-\frac{1}{8}\alpha^2+\frac{1}{2}\beta+i\gamma-\frac{1}{2}\delta+\frac{1}{2}\alpha p\right) dx\wedge dy+O(t^3)\\
=&b i+t^2\frac{i}{4\pi^2}\int_{T^2}\left(\frac{1}{2}\Re{\alpha p}+\frac{u'}{v}\kappa \varphi^2+\frac{v'^2}{2v^2}\varphi^2-\frac{1}{2}\varphi_y^2+\frac{v'}{v}\varphi\varphi_x+\frac{1}{2}\varphi_x^2\right) dx\wedge dy\\
&-t^2\frac{1}{4\pi^2}\int_{T^2}\left(\frac{1}{2}\Im{\alpha p}+\frac{v'}{v}\varphi\varphi_y+\varphi_x\varphi_y\right) dx\wedge dy +O(t^3).
\end{split}\end{equation}

Note that only the real part is relevant for us as we are only  interested in the imaginary part of the change 
of the modulus $\tau$, and that the functions $u,v,\kappa,\varphi$ are real-valued. 
We still need to analyze
$$\int \tfrac{1}{2}\alpha p \; dx\wedge dy.$$

Recall that $p$ is the unique solution of $\bar\partial p=-i\frac{\alpha_y}{4} d\bar z$
perpendicular to the constants. 

Hence, for
\[\alpha=-2\kappa\varphi=\sum_{k,\tfrac{l}{b}\in\N}  \left(a^{kl}_1\cos{l x}\cos{k y}+a^{kl}_2\cos{l x}\sin{k y}+a^{kl}_3\sin{l x}\cos{k y}+a^{kl}_4\sin{l x}\sin{k y}\right)\]
we have
\[p=\frac{1}{2}\nu_x-\frac{i}{2}\nu_y\]
where
 \[\nu=\sum_{k,l\cdot b\in\N; k>0}\left(-\frac{i k}{k^2+l^2}(a^{kl}_1\cos{l x}\sin{k y}-a^{kl}_2\cos{l x}\cos{k y}+a^{kl}_3\sin{l x}\sin{k y}-a^{kl}_4\sin{l x}\cos{k y})\right).\]

Using Fubini and integration by parts
together with 
\[\int_{T^2}\cos^2{l x}\cos^2{k y} dx\wedge dy= \int_{T^2}\sin^2{l x}\sin^2{k y} dx\wedge dy=\pi^2 b\]
(independently of $k\in\N, l/b\in\N$), we obtain
\begin{equation}
\begin{split}
\Re\left(\int_{T^2}\frac{1}{2}\alpha p \; dx\wedge dy\right)
=-\frac{\pi^2b}{4}\sum_{k,\tfrac{l}{b}\in\N; k>0}\frac{k^2}{k^2+l^2} \left((a_1^{kl})^2+(a_2^{kl})^2+(a_3^{kl})^2+(a_4^{kl})^2\right). \end{split}\end{equation}

Putting all terms together we have
\[\left(\frac{d}{dt}\right)^2{\mid _{t=0}}\Pi^2(F_t)= \int_{T^2}\left(\mathcal B + \mathcal Q_2(\varphi)\right) \varphi \; dx\wedge dy,\]
for some $\mathcal B$ tensorial in $y$ and with 
\[\mathcal Q_2(\varphi) =  \tfrac{1}{2} \varphi_{yy} +  \tfrac{1}{2}K,\] 
where \[<K(\varphi),\varphi>_{L^2}=\int_{T^2} \Re (\alpha p(\varphi))=-2\int_{T^2}\kappa\varphi \Re(p(\varphi)),\] $\alpha=-2\kappa \varphi$, and $p=p(\varphi)$ is determined by \eqref{dbarpay}.

\begin{Rem}
Note that  
$$\delta^2\mathcal E_{\beta^b}=\mathcal A+ \mathcal B.$$ 
\end{Rem}
\subsection{Final estimates}

\begin{Lem}\label{lneq1}
Let $f^b$ be the 2-lobed Delaunay torus of conformal type $(0,b)$ and 
$$\mathcal Q= \left ( \tfrac{1}{8}\kappa^2 +  \tfrac{1}{2}\right )\del^2_{yy} + \tfrac{1}{4}\del^4_{yyyy} +\tfrac{1}{2} \del^4_{xxyy} - \beta^b \left( \tfrac{1}{2} \varphi_{yy} + \tfrac{1}{2}K \right) ,$$
where 
$\kappa= \kappa(x)$  is the geodesic curvature of the profile curve $\gamma^b.$ Then we have for $\varphi : T^2 \rightarrow \R$
$$\langle \mathcal Q(\varphi), \varphi \rangle_{L^2} \;  > \;0$$ if $\varphi(x_0,y)\perp_{L^2} \sin(y), \cos(y)$ for all $x_0 \in \R/ 2 \pi b \Z.$
\end{Lem}
\begin{proof}
Since $\langle K(\varphi), \varphi \rangle_{L^2}  \leq0$ and $\beta^b >0,$ it suffices to show the positivity of 
$$\tilde {\mathcal Q} =  \left ( \tfrac{1}{8}\kappa^2 +  \tfrac{1}{2}- \tfrac{1}{2}\beta^b  \right)\del^2_{yy} + \tfrac{1}{4}\del^4_{yyyy} +\tfrac{1}{2} \del^4_{xxyy}.$$

We expand $\varphi$ with respect to the Fourier basis along the imaginary period of $T^2,$ i.e.,
\[\varphi=\sum_{l\in\Z} \varphi_l(x) \cos( ly)+\tilde \varphi_l(x) \sin( ly)\]
for periodic real-valued functions $\varphi_l$ and $\tilde \varphi_l.$
In $y$, the Fourier basis elements are perpendicular with respect to $\langle \tilde{\mathcal Q}(.), .\rangle_{L^2}$, therefore it suffices to show positivity for elements of the form  
\[ \varphi= \varphi_l(x)\sin(ly),\]

for which we compute:

\begin{equation}\langle\tilde { \mathcal Q}(\varphi), \varphi \rangle = \int_0^L(-\left ( \tfrac{1}{8}\kappa^2 +  \tfrac{1}{2} - \beta^b  \tfrac{1}{2}\right )l^2 + \tfrac{1}{4} l^4)\varphi_l^2 dx + \tfrac{1}{2}l^2\int_0^L (\del_x\varphi_l(x))^2 dx.
\end{equation}

In order to show positivity, we first need a point-wise estimate on $\kappa^2.$ 
The profile curve  $\gamma^b$ is an orbit-like elastic curve, i.e., it satisfies the ODE
$$\kappa'' + \tfrac{1}{2} \kappa^3 + (\mu- 1) \kappa =  0$$
with $\mu = -  \beta^b$, and the 4th order polynomial
$$P_4 := \tfrac{1}{4} \kappa^4 + (\mu -1) \kappa^2 + \nu$$
has $4$ real roots. For $\mu$ fixed the value of $\kappa$ is bounded by the biggest of the 4 roots of $P_4$. This gives

$$\kappa^2 < 4 - 4\mu.$$
Therefore, we obtain for $l \geq2$ positivity of $\tilde{\mathcal Q}.$
\end{proof}

\begin{Lem}
Let $f^b$ be the 2-lobed Delaunay torus of conformal type $(0,b)$, then  the
pseudo differential operator $\mathcal Q$ has a 4-dimensional kernel.
\end{Lem}
\begin{proof}
Without loss of generality we restrict to the space of functions
\[\{\varphi = f(x) \cos(y)+g(x) \sin(y)\mid f,g\in W^{4,2}(\R/2\pi b) \}.\]
Define
$$\mathcal Q^1:=  -\left ( \tfrac{1}{8}\kappa^2 +  \tfrac{1}{2}\right )+ \tfrac{1}{4} -\tfrac{1}{2} \del^2_{xx} - \beta^b \left( -\tfrac{1}{2}  + K\right). $$

For $\varphi=f(x) \cos(y)+g(x) \sin(y)$
$$ \varphi \in  \text{Ker}\mathcal Q$$ is equivalent to
$$f, g \in \text{Ker} \mathcal Q^1,$$ and it remains to show that $\text{Ker}\mathcal Q^1$ is 2-dimensional. We consider the equation
\[\mathcal Q^1(f)=0, \quad \text{with}\quad  f=f(x) \in W^{4,2}(\R/2\pi b). \]

Since $\kappa >0$ for orbit-like curves, we can divide by $\kappa$ and consider the equation
\[0=\frac{1}{\kappa}\left ( -\tfrac{1}{8}\kappa^2 -  \tfrac{1}{4} +\tfrac{1}{2} \del^2_{xx}-\tfrac{1}{2}\beta^b\right )f + \beta^b P(-2\kappa f),\]
with 
\[P\left( \sum_{l,\; l \cdot b \in \N} a_l \cos(lx)+b_l\sin(l x)\right)= -\sum_{l,\; l \cdot b \in \N} \tfrac{1}{1+l^2}(a_l \cos(lx)+b_l\sin(l x))\]
instead.

The operator 
\[h\in W^{4,2}(S^1)\longmapsto \left(\frac{d^2}{dx^2}-1\right) h\in W^{2,2}(S^1)\]
is injective and surjective (by the Fredholm alternative). Therefore, for $f\in $ Ker$\mathcal Q^1$ there exits a unique $h\in W^{4,2}$ such that
\[\left(\frac{d^2}{dx^2}-1\right) h=-2\kappa f.\]
Then,
\begin{equation}\label{ode4p}
\begin{split}
0=&\frac{1}{\kappa}\left ( -\tfrac{1}{8}\kappa^2 -  \tfrac{1}{4} +\tfrac{1}{2} \del^2_{xx}-\tfrac{1}{2}\beta^b\right )f + \beta^b P(-2\kappa f)\\
=&-\frac{1}{\kappa}\left( ( -\tfrac{1}{8}\kappa^2 -  \tfrac{1}{4} +\tfrac{1}{2} \del^2_{xx}-\tfrac{1}{2}\beta^b\right ) \tfrac{1}{2\kappa})\circ\left(\frac{d^2}{dx^2}-1\right) h+ \beta^b P\circ \left(\frac{d^2}{dx^2}-1\right) h.
\end{split}
\end{equation}
A short computation shows that  \[ P\circ \left(\frac{d^2}{dx^2}-1\right)= \text{id}.\]
Therefore, \eqref{ode4p} is a 4th order ordinary differential equation (on the function $h$). Thus, when ignoring periodicity,
its solution space is $4$-dimensional. Two of which are given by non-periodic solutions and spanned by
$$h = \sinh(x) \quad h = \cosh(x).$$ 

Hence, at most a 2-dimensional subspace of solution is well-defined on $\R/2\pi b$.
By considering the infinitesimal Moebius transformations we obtain that the space of global solutions is exactly 2-dimensional, proving  the lemma.
\end{proof}
\begin{proof}[Proof of Theorem 3]
Due to Lemma \ref{lneq1} and the actual form of the second variation operator it remains to consider the case $l=1.$ The space of infinitesimal M\"obius variations normal to $f^b$ is $9$-dimensional for $b > b_0,$ consisting of a 3-dimensional subspace 
of 0 Fourier mode in $y$, and a 6-dimensional subspace of Fourier mode 1 in $y.$
This 6-dimensional space corresponds to the non-positive directions of $\mathcal Q$ as can be seen as follows: For $b< b_0$ a straight forward computation shows that the kernel of $\mathcal Q$ is $4$ dimensional and it has $2$ negative directions, all of which giving rise to M\"obius variations. Because the kernel dimension of $\mathcal Q$ remains $4$-dimensional for all $f^b$ by the previous lemma
 (and because of the spectral properties of $\mathcal Q$) the space of non-positive directions of $\mathcal Q$ remains $6$ dimensional for all $f^b,$ which corresponds exactly to the M\"obius variations.
\end{proof}

%%%%%%%%%%%%%%%%%%%%%%%%%%%%%%%%%%%%%%%%%%%%%%%%%%%%%%%%%%%%%%%%%%%%%%%%%%%%%%
%           Bibliography                                                     %
%%%%%%%%%%%%%%%%%%%%%%%%%%%%%%%%%%%%%%%%%%%%%%%%%%%%%%%%%%%%%%%%%%%%%%%%%%%%%%

\bibliographystyle{amsplain}
\bibliography{references}

\end{document}